\documentclass[12pt]{amsart}
\usepackage{amssymb, enumitem}
 
\usepackage[all]{xypic}

\newcommand{\np}{\vspace{3mm}\refstepcounter{subsection}\noindent{\bf \thesubsection.}
}

\numberwithin{equation}{section}

\newcommand{\Hom}{\operatorname{Hom}}
\renewcommand{\geq}{\geqslant}
\renewcommand{\leq}{\leqslant}
\newcommand{\Osh}{{\mathcal O}}                        

\renewcommand{\H}{\mathrm{H}}                          
\newcommand{\cchar}{\operatorname{char}}
\newcommand{\ii}{\operatorname{i}}
\newcommand{\NS}{\operatorname{NS}} 

\newcommand{\End}{\operatorname{End}} 
\newcommand{\kk}{\mathbf{k}}
\newcommand{\Tr}{\operatorname{Tr}}
\newcommand{\Trd}{\operatorname{Trd}}
\newcommand{\Nr}{\operatorname{Nr}}
\newcommand{\Nrd}{\operatorname{Nrd}}
\newcommand{\N}{\operatorname{N}}

\renewcommand{\emptyset}{\varnothing}
\newcommand{\CC}{\mathbb{C}} 
\newcommand{\QQ}{\mathbb{Q}} 
\newcommand{\RR}{\mathbb{R}} 
\newcommand{\ZZ}{\mathbb{Z}} 

\newtheorem{theorem}{Theorem}[section]
\newtheorem{lemma}[theorem]{Lemma}
\newtheorem{corollary}[theorem]{Corollary}
\newtheorem{proposition}[theorem]{Proposition}
\theoremstyle{definition}

\begin{document}
\title[Reduced norms and the Riemann-Roch theorem]{Reduced norms and the Riemann-Roch theorem for Abelian varieties}

\author{Nathan Grieve}
\address{Department of Mathematics and Statistics,
University of New Brunswick,
Fredericton, NB, Canada}
\email{n.grieve@unb.ca}%

\begin{abstract} We explain how the Riemann-Roch theorem for divisors on an abelian variety $A$ is related to the reduced norms of the Wedderburn components of $\End^0(A)$ the $\QQ$-endomorphism algebra of $A$.  We then describe consequences and examples.
\end{abstract}

\maketitle

\section{Introduction}\label{1}
Let $A$ be an abelian variety defined over a field $\kk$.  Our purpose here is to show how the Riemann-Roch theorem for divisors on $A$ is related to the reduced norms of the Wedderburn components of 
$$R := \End^0(A) := \QQ \otimes_\ZZ \End(A)$$ 
the $\QQ$-endomorphism algebra of $A$.  We then turn to applications and examples.

The key point to what we do here is Theorem \ref{theorem4.1} and a special case of this result applies to principally polarized abelian varieties (ppav).  To describe it,  suppose that 
\begin{equation}\label{eqn1.1}
A=A_1^{r_1} \times \dots \times A_k^{r_k}
\end{equation}
with each of the $A_i$ simple and pairwise nonisogenous abelian varieties.  Assume further that $A$ admits an ample divisor $\lambda$ which defines a principal polarization of $A$. Let $r_\lambda : R \rightarrow R$ denote the Rosati involution determined by $\lambda$.  Since $\lambda$ is a principal polarization, it determines an isomorphism 
\begin{equation}\label{eqn1.2}
\Phi_\lambda : \NS(A) \xrightarrow{\sim} \End_\lambda(A)
\end{equation}
between the N\'{e}ron-Severi group of $A$ and the subgroup 
$$
\End_\lambda(A) := \{ \alpha \in \End(A) : r_\lambda (\alpha) = \alpha \} 
$$
of $\End(A)$ fixed by $r_\lambda$.

The decomposition \eqref{eqn1.1} allows us to write 
$$
R = R_1\times \dots \times R_k
$$
with $R_i \simeq M_{r_i}(\Delta_i)$, for finite dimensional division rings $\Delta_i$ over $\QQ$.  In what follows we let $Z_i$ denote the centre of $\Delta_i$, $m_i^2 := \dim_{Z_i} \Delta_i$, $t_i :=[Z_i:\QQ]$, $g_i := \dim A_i$ and $g := \dim A$.  We also let $\Nrd_{R_i/\QQ}(\cdot)$ denote the reduced norm from $R_i$ to $\QQ$.
In this notation, a consequence of Corollary \ref{corollary3.7} is that the function
$$ \prod_{i=1}^k \Nrd_{R_i/\QQ}(\cdot)^{2g_i/(t_im_i)} : R \rightarrow \QQ$$
restricted to $\End^0_\lambda(A) := \QQ \otimes_\ZZ \End_\lambda(A)$ is the square of a homogeneous polynomial function of degree $g$ on $\End_\lambda^0(A)$ normalized so as to take value $1$ when evaluated at $1_R$; we denote this polynomial function by 
$$\mathrm{pNrd}_\lambda(\cdot) : \End^0_\lambda(A) \rightarrow \QQ$$ in what follows.
In this setting, a consequence of Theorem \ref{theorem4.1} is:

\begin{theorem}\label{theorem1.1}
If $D$ is a divisor on the ppav $A$ just described and $\alpha = \Phi_\lambda(D) \in \End_\lambda(A)$, the image of the class of $D$ under \eqref{eqn1.2}, then 
\begin{equation}\label{1.1} \frac{(D^g)}{g!} = \mathrm{pNrd}_\lambda(\alpha). 
\end{equation}
\end{theorem}

We prove Theorem \ref{theorem1.1} in \S \ref{4}.  To place this theorem and the more general Theorem \ref{theorem4.1} into their proper context we should first emphasize that these theorems are due to Mumford, \cite[\S 21, p.~209]{Mum:v1}.  Our contribution here is to make them slightly more explicit and then to indicate some consequences, for example Theorems \ref{theorem4.4} and \ref{theorem5.1}.  A related question, which we address to some extent in \S \ref{5}, \S \ref{6}, and \S \ref{7}, is to understand the righthand side of \eqref{1.1} explicitly for a given $A$ and principal polarization $\lambda$.  To place these matters into perspective we recall that, especially when $\cchar \kk > 0$, it is already an interesting and subtle problem to understand those semisimple $\QQ$-algebras that are actually realized as $\End^0(A)$ for some abelian variety $A$; see \cite{Shimura}, \cite{Oort:1988} and \cite[Cor.~1.2]{Oort:VanDerPut:1988} for example, for more details and very complete results in this direction.

\subsection*{Acknowledgements}
This work benefited from correspondence with Ernst Kani and conversations with Colin Ingalls.  It was  completed while I was a postdoctoral fellow at the University of New Brunswick where I was financially supported by an AARMS postdoctoral fellowship.  Finally, I thank an anonymous referee for  carefully reading this work and for providing comments and suggestions which helped to improve the exposition and strengthen the results.

\section{Preliminaries}

We fix some notation and recall some prerequisites which are important in what follows.  We refer to \cite{Mum:v1}, \cite{CRI}, \cite{CRII}, and \cite{Reiner:2003}, for example, for more details.

\np\label{Sec2.1}
If $R$ is a simple finite dimensional $\QQ$-algebra with centre $Z$, then $R\simeq M_{r}(\Delta)$ for $\Delta$ a finite dimensional division ring over $\QQ$ and we let $m$ denote the \emph{Schur index} of $\Delta$.  Then $m^2 := \dim_{Z(\Delta)} \Delta$ for $Z(\Delta) \simeq Z$ the centre of $\Delta$, \cite[Cor.~7.22]{CRI} and \cite[p.~732]{CRII}.  Put $t := [Z:\QQ]$.  We then have that $\dim_Z R = r^2 m^2 =: d^2$ and also that $\dim_\QQ R = td^2$. 

\np\label{2.1'} We recall the definition of the reduced norm and reduced trace from a finite dimensional simple $\QQ$-algebra $R \simeq M_r(\Delta)$ to $\QQ$.  Our approach is similar to that of \cite[\S 9.a and \S 9.b]{Reiner:2003}.  To begin with, let $Z$ be the centre of $\Delta$ and choose an extension field $F/Z$ which \emph{splits} $R$.  More precisely, choose an extension field $F/Z$ for which there is an isomorphism
$F\otimes_Z R \xrightarrow{\sim} M_{rm}(F)$
of $F$-algebras.  For example, if $F$ is a maximal subfield of $\Delta$, then $F$ is a splitting field for $R$, \cite[Cor.~7.21, p.~155]{CRI}.
Next if $\alpha \in R$, then let $\Nrd_{R/Z}(\alpha)$ and $\Trd_{R/Z}(\alpha)$ denote, respectively, the determinant and trace of the image of $1\otimes\alpha$ under this isomorphism.  Recall that $\Nrd_{R/Z}(\alpha)$ and $\Trd_{R/Z}(\alpha)$ are well-defined elements of $Z$, \cite[Cor.~7.23, Thm.~ 9.3]{Reiner:2003}.  We say that the functions $\Nrd_{R/Z}(\cdot)$ and $\Trd_{R/Z}(\cdot)$ are, respectively, the \emph{reduced norm} and \emph{reduced trace} from $R$ to $Z$.
Finally, to define the reduced norm and reduced trace from $R$ to $\QQ$, we let $\Nr_{Z/\QQ}(\cdot)$ and $\Tr_{Z/\QQ}(\cdot)$ denote, respectively, the norm and trace from $Z$ to $\QQ$.  Concretely, if $\sigma_1,\dots ,\sigma_t$ are the distinct embeddings of $Z$ into $\CC$ and $\alpha \in Z$, then 
$\Nr_{Z/\QQ}(\alpha) := \prod_{i=1}^t \sigma_i(\alpha)$ 
and 
$\Tr_{Z/\QQ}(\alpha) := \sum_{i=1}^t \sigma_i(\alpha)$.  The \emph{reduced norm} and \emph{reduced trace} from $R$ to $\QQ$ are denoted, respectively, by $\Nrd_{R/\QQ}(\cdot)$ and $\Trd_{R/\QQ}(\cdot)$.  They are defined by composition: 
$\Nrd_{R/\QQ}(\cdot) := \Nr_{Z/\QQ}\circ \Nrd_{R/Z}(\cdot)$ 
and 
$\Trd_{R/\QQ}(\cdot) := \Tr_{Z/\QQ} \circ \Trd_{R/Z}(\cdot)$.

\np\label{Sec2.2} If $R$ is a finite dimensional semisimple $\QQ$-algebra with Wedderburn decomposition $R = R_1 \times \dots \times R_k$, for matrix algebras $R_i \simeq M_{r_i}(\Delta_i)$ over finite dimensional division rings $\Delta_i$ with centres $Z_i$, then we let $m_i$ denote the Schur index of $\Delta_i$, $t_i := [Z_i : \QQ]$, and $d_i := m_i r_i$.  As in \cite[Defn.~ 3.23]{CRI}, we refer to the simple $\QQ$-algebras $R_i$ as the \emph{Wedderburn components} of $R$; they are uniquely determined by $R$.  The \emph{reduced norm} and \emph{reduced trace}, respectively, from $R$ to $\QQ$ are defined as: 
$\Nrd_{R/\QQ}(\cdot) := \prod_{i=1}^k \Nrd_{R_i / \QQ}(\cdot)$ and 
$\Trd_{R/\QQ}(\cdot) := \sum_{i=1}^k \Trd_{R_i/\QQ}(\cdot)$, \cite[p.~121]{Reiner:2003}.  

\np\label{Sec2.2.1} Let $R$ be a finite dimensional semisimple $\QQ$-algebra.  By a \emph{positive involution} on $R$, we mean a $\QQ$-linear map $':R\rightarrow R$ which satisfies the conditions that $(\alpha \beta)' = \beta' \alpha'$ for all $\alpha,\beta \in R$, $(\alpha')' = \alpha$ for all $\alpha \in R$ and $\Trd_{R/\QQ}(\alpha \alpha') > 0$ for all nonzero $\alpha \in R$, \cite[\S 21, p.~193]{Mum:v1}.   By a \emph{norm form} on $R$ over $\QQ$, we mean a non-zero polynomial function $\N(\cdot) : R \rightarrow \QQ$ such that $\N(\alpha \beta) = \N(\alpha)\N(\beta)$ for each $\alpha,\beta \in R$, \cite[\S 19, p.~178]{Mum:v1}.

\np\label{Sec2.3} We consider abelian varieties defined over a field $\kk$.  Recall that a non-zero abelian variety is said to be \emph{simple} if it contains no proper non-zero abelian subvarieties.  In what follows, all abelian varieties will be tacitly assumed to be non-zero.  If $A$ is an abelian variety, then we let $\hat{A}$ denote its dual and we let $R := \End^0(A) := \QQ\otimes_\ZZ \End(A)$ denote the $\QQ$-endomorphism algebra of $A$.  Recall that $R$ is a finite dimensional semisimple $\QQ$-algebra and, further, that $R$ is a division algebra when $A$ is simple, \cite[\S 19, p.~174]{Mum:v1},  \cite[Cor.~3.20]{Conrad:2006}, 
\cite[Thm.~1.2.1.3]{Chai:Conrad:Oort:2014}.  If $a \in A(\kk)$, then $\tau_a : A \rightarrow A$ denotes translation by $a$ in the group law.  Every divisor $D$ on $A$ determines a homomorphism $\phi_D : A \rightarrow \hat{A}$ defined by $a \mapsto \tau_a^* \Osh_A(D) \otimes \Osh_A(-D)$, \cite[\S 13, Cor.~5, p.~132]{Mum:v1}.  If $\phi \in \End^0(A)$, then we let $\hat{\phi} \in\End^0(\hat{A})$ denote its dual, \cite[\S 8, Rmks.,~p.~80]{Mum:v1}.  If $\lambda$ is an ample divisor on $A$, then we let $r_\lambda : R\rightarrow R$ denote the Rosati-involution.  Recall that $r_\lambda$ is defined by $\alpha \mapsto \phi_\lambda^{-1} \circ \hat{\alpha} \circ \phi_\lambda$ and that $r_\lambda$ is a positive involution, \cite[\S 21, Thm.~1, p.~192]{Mum:v1}, \cite[Lem.~1.3.5.3]{Chai:Conrad:Oort:2014}.  Here $\phi_\lambda^{-1}$ is the inverse of $\phi_\lambda$ in the category of abelian varieties up to isogeny.  If $n \in \ZZ$, then we let $[n]_A : A \rightarrow A$ denote multiplication by $n$ in the group law.

\np\label{Sec2.4} Let $A$ be an abelian variety of dimension $g$. We denote the $g$-fold self-intersection number of a divisor $D$ on $A$ by $(D^g)$.  Let $\NS(A)$ be the N\'{e}ron-Severi group of divisors on $A$ modulo algebraic equivalence and $\NS^0(A) := \QQ \otimes_\ZZ \NS(A)$ the N\'{e}ron-Severi space of $A$.  Note that the map $D \mapsto \phi_D$ induces an inclusion $\NS(A) \hookrightarrow \Hom(A,\hat{A})$.  In particular, $\NS(A)$ is free, \cite[\S 19, Cor.~2, p.~178]{Mum:v1}, \cite[Cor.~12.8]{Milne:1986}, and the natural map $\NS(A) \rightarrow \NS^0(A)$ is injective.  By abuse of terminology, we refer to elements of $\NS^0(A)$ as $\QQ$-divisors and elements of $\NS(A)$ as integral divisors.  By abuse of notation, we denote the class of a divisor $D$ on $A$ in $\NS(A)$ also by $D$.  

\np\label{Sec2.4.1} Every ample divisor $\lambda$ on an abelian variety $A$ determines a $\QQ$-linear embedding 
$$\Phi_\lambda : \NS^0(A) \hookrightarrow \End^0(A)$$ 
defined by 
$D \mapsto \phi_\lambda^{-1} \circ \phi_D.$   
As in \cite[\S 20, Cor.~, p.~191]{Mum:v1}, see also \cite[Chap.~XII]{MV}, we check, using the theory of Riemann-forms, that the image of $\Phi_\lambda$ is the $\QQ$-vector space 
$$
\End^0_\lambda(A) := \{ \alpha \in \End^0(A) : r_\lambda(\alpha) = \alpha \}
$$ 
of elements of $\End^0(A)$ which are fixed by $r_\lambda$.  Note that if $\lambda$ is a \emph{principal polarization}, that is if $(\lambda^g) = g!$, then $\Phi_\lambda$ determines a $\ZZ$-linear isomorphism $\Phi_\lambda : \NS(A) \xrightarrow{\sim} \End_\lambda(A)$. Here $\End_\lambda(A) := \{ \alpha \in \End(A) : r_\lambda(\alpha) = \alpha \}$.    When $\lambda$ defines a principal polarization, we say that the pair $(A,\lambda)$ is a \emph{principally polarized abelian variety}.

\section{Isogenies and the degree function}\label{3}

In this section we record some facts about isogenies between abelian varieties and the degree function.  The main point is to establish Proposition \ref{theorem3.5} and its corollaries which we need in \S \ref{4}.

\np\label{} Let $\alpha : B_1 \rightarrow B_2$ be a homomorphism of abelian varieties. Recall, \cite[p.~114]{Milne:1986} or \cite[\S 6, p.~63]{Mum:v1}, that $\alpha$ is an \emph{isogeny} if it is surjective and has kernel a finite group scheme.  If $\alpha$ is an isogeny, then we define its \emph{degree}, which we denote by $\deg (\alpha)$, to be the order of its kernel (as a finite group scheme).  Precisely, $\deg(\alpha)$ equals the rank of $\alpha_* \Osh_{B_1}$ as a locally free $\Osh_{B_2}$-module.  This concept can be extended to all homomorphisms from $B_1$ to $B_2$.  In particular, we have the degree function
\begin{equation}\label{3.1}
\deg(\cdot) : \Hom(B_1,B_2) \rightarrow \ZZ
\end{equation}
defined by
\begin{equation}\label{3.2}
\deg (\alpha) := \begin{cases}
\deg (\alpha) & \text{ if $\alpha$ is an isogeny} \\
0 & \text{ otherwise.}
\end{cases}
\end{equation}
Such functions also have the properties that
\begin{equation}\label{3.3}
\deg(\beta \circ \alpha) = \deg(\beta) \deg(\alpha)  
\end{equation} for all triples of  abelian varieties, $B_1$, $B_2$, and $B_3$, and homomorphisms  
$\alpha \in \Hom(B_1,B_2)$ and 
$\beta \in \Hom(B_2,B_3),$ \cite[p.~115]{Milne:1986}.

\np\label{extended:degree:properties}
When $A = B_1 = B_2$, the function 
$\deg(\cdot) : \End(A) \rightarrow \ZZ$ extends to a degree function 
$\deg(\cdot) : \End^0(A) \rightarrow \QQ $.  Further, if $g = \dim A$, then this function 
is a homogeneous polynomial function of degree $2g$, \cite[\S19, Thm.~2, p.~174]{Mum:v1} or  \cite[Prop.~12.4, p.~123]{Milne:1986}.
The degree function \eqref{3.1} also extends to give a degree function
\begin{equation}\label{3.4}
\deg(\cdot): \Hom^0(B_1,B_2) := \QQ \otimes_{\ZZ} \Hom(B_1,B_2)\rightarrow \QQ.
\end{equation}
Indeed, given $\alpha \in \Hom(B_1,B_2)$ and $n \in \ZZ$, the morphism $n\alpha \in \Hom(B_1,B_2)$ is defined by $x \mapsto n \alpha(x)$.  In particular,
$ n \alpha = [n]_{B_2} \circ \alpha$
and so
$$ \deg(n \alpha) = \deg([n]_{B_2} \circ \alpha) = \deg([n]_{B_2}) \deg(\alpha).$$
Now, given $\alpha \in \Hom^0(B_1,B_2)$, let $n$ be an integer so that $n\alpha$ is an element of $\Hom(B_1,B_2)$ and define
\begin{equation}\label{3.5}
 \deg (\alpha) := \frac{\deg(n \alpha)}{\deg ([n]_{B_2})}.
 \end{equation}
The definition \eqref{3.5} is well defined and extends that given in \eqref{3.2}.  The extended degree functions \eqref{3.4} also behave well with respect to composition of morphisms as in \eqref{3.3}.

\np\label{} In what follows, we fix an ample divisor $\lambda$ on an abelian variety $A$.  We record the following for later use.

\begin{lemma}\label{lemma3.1}
If $D$ is a divisor on $A$, we then have that
\begin{equation}\label{3.6} \deg(\phi_D) = \frac{\deg(\Phi_{\lambda}(D))}{\deg(\phi_{\lambda}^{-1})}=\deg(\phi_\lambda)\deg(\Phi_{\lambda}(D)).
\end{equation}
\end{lemma}
\begin{proof}
Since $\Phi_\lambda(D) = \phi^{-1}_\lambda \circ \phi_D$, in light of the formula \eqref{3.3}, it suffices to note that 
$$ \deg(\Phi_{\lambda}(D)) = \deg(\phi_\lambda^{-1})\deg(\phi_D).$$
\end{proof}

\np\label{S:2.4}  Every isogeny $\beta : B \rightarrow A$ induces a ring isomorphism 
\begin{equation}\label{3.7} [\beta] : \End^0(B) \xrightarrow{\sim} \End^0(A)
\end{equation}
defined by 
$$
[\beta](\alpha) = \beta \circ \alpha \circ \beta^{-1}.
$$
Here $\beta^{-1}$ denotes the inverse of $\beta$ in the category of abelian varieties up to isogeny.  Specifically, if $n = \deg \beta$ and $\gamma$ is the unique isogeny $A \rightarrow B$ so that $\gamma \circ \beta = [n]_B$ and $\beta \circ \gamma = [n]_A$, then $\beta^{-1}$ is the quasi-isogeny defined by $\beta^{-1} = (1/n) \gamma$, \cite[\S19, p.~172]{Mum:v1}.  The homomorphism \eqref{3.7} behaves well with respect to the degree homomorphism.  In particular
\begin{equation}\label{3.9}
\deg([\beta](\alpha)) = \deg(\beta)\deg(\alpha) \deg(\beta^{-1}) = \deg(\alpha),
\end{equation}
for all $\alpha \in \End^0(B)$.

If we fix an ample divisor $\lambda$ on $A$, then functorial properties of the isomorphism $\Phi_\lambda$ with respect to $\beta$ and the inverse of \eqref{3.7} are explained in \cite[p.~50]{Kani:2016}.  Here we record the following related remark which we will need in what follows.

\begin{proposition}\label{prop3.2}
Let $\beta : B \rightarrow A$ denote an isogeny of abelian varieties, fix an ample divisor $\lambda$ on $A$ and let  
$$r_{\beta^* \lambda} : \End^0(B) \rightarrow \End^0(B)$$ 
be the Rosati involution determined by the ample divisor $\beta^*\lambda$ on $B$.  Then if $\alpha \in \End^0(B)$, it follows that 
\begin{equation}\label{2.10}
[\beta](r_{\beta^* \lambda} \alpha) = r_\lambda([\beta](\alpha)).
\end{equation}
\end{proposition}
\begin{proof}
We compute 
$$[\beta](r_{\beta^* \lambda} \alpha) = \beta \circ \phi_{\beta^*\lambda}^{-1} \circ \hat{\alpha} \circ \phi_{\beta^* \lambda} \circ \beta^{-1}$$ whereas 
$$r_\lambda([\beta](\alpha)) = 
\phi^{-1}_\lambda \circ \widehat{\beta^{-1}} \circ \widehat{\alpha} \circ \widehat{\beta} \circ \phi_\lambda.$$  The conclusion of the proposition will follow if we can show that 
\begin{equation}\label{2.11}
\hat{\beta} \circ \phi_\lambda = \phi_{\beta^* \lambda} \circ \beta^{-1}.
\end{equation}
To establish \eqref{2.11}, it suffices to show that
$$
\hat{\beta} \circ \phi_\lambda \circ \beta = \phi_{\beta^* \lambda};
$$
equivalently that
\begin{equation}\label{2.12'}
\beta^*(\tau_{\beta(b)}^* \Osh_A(\lambda) \otimes \Osh_A(-\lambda)) \simeq \tau_b^* \beta^* \Osh_A(\lambda)\otimes \beta^* \Osh_A(-\lambda),
\end{equation}
for all $b \in B$.
That \eqref{2.12'} holds can be checked using the fact that $\beta \circ \tau_b = \tau_{\beta(b)} \circ \beta$ for all $b \in B$.
\end{proof}

Proposition \ref{prop3.2} has the following consequence.

\begin{corollary}\label{cor3.3}
If $\beta : B \rightarrow A$ is an isogeny and $\lambda$ an ample divisor on $A$, then the morphism $[\beta]$ induces the following collection of linear maps
\begin{equation}\label{3.13}
\NS(B) \hookrightarrow \NS^0(B) \xrightarrow{\sim} \End^0_{\beta^* \lambda}(B) \xrightarrow{\sim} \End^0_\lambda(A).
\end{equation}
\end{corollary}
\begin{proof}
Immediate consequence of \eqref{2.10} and the definitions of $\End^0_{\beta^* \lambda}(B)$ and $\End^0_\lambda(A)$.
\end{proof}

For later use we also record:
\begin{lemma}\label{lemma3.4}
Let $A = A_1^{r_1} \times \dots \times A_{k}^{r_k}$ with the abelian varieties $A_i$ simple and pairwise nonisogenous.  If $\beta : B \rightarrow A$ is an isogeny, if $\alpha \in \End^0(B)$ and $[\beta](\alpha) = (\alpha_1,\dots, \alpha_k)$ with $\alpha_i \in \End^0(A_i^{r_i})$, then 
\begin{equation}\label{3.14}
\deg(\alpha) = \prod_{i=1}^k \deg(\alpha_i).
\end{equation}
\end{lemma}
\begin{proof}
Since, by \eqref{3.9}, the degree function behaves well with respect to the ring isomorphism $[\beta]$, without loss of generality we may assume that $A = B$ and that $\beta = \operatorname{id}_A$.  Also, in light of the definition \eqref{3.5}, without loss of generality we may assume that $\alpha \in \End(A)$.  If $\alpha$ is not an isogeny, then some $\alpha_i$ is also not an isogeny and the formula holds.  On the other hand, if $\alpha$ is an isogeny, then each of the $\alpha_i$ must be an isogeny and clearly 
$$ \deg (\alpha) = \prod_{i=1}^k \deg (\alpha_i).$$
\end{proof}

\np\label{Sec2.5}  We use our remarks made thus far to prove the main results of this section.  To prepare for these statements, let $A = A_1^{r_1} \times \dots \times A_k^{r_k}$ with each $A_i$ simple abelian varieties and pairwise nonisogenous.  Let $g_i := \dim A_i$ and write $R := \End^0(A) = R_1 \times \dots \times R_k$ with $\Delta_i = \End^0(A_i)$ and $R_i = M_{r_i}(\Delta_i)$; we adopt the notation of \S \ref{Sec2.2}.  

The following remark, and its corollary, should be compared with \cite[\S 19, Cor., p.~182]{Mum:v1}.

\begin{proposition}\label{theorem3.5}
In the setting of \S \ref{Sec2.5} just described, the following assertions hold true:
\begin{enumerate}
\item{$2 g_i / (t_i m_i) \in \ZZ$, for $i=1,\dots, k$;}
\item{$\deg(\alpha) = \prod_{i=1}^k \Nrd_{R_i/\QQ}(\alpha_i)^{2g_i/(t_i m_i)}$.}
\end{enumerate}
\end{proposition}
\begin{proof}
Because of the multiplicative nature of the degree function, namely \eqref{3.14}, it suffices to show that, for $i$ fixed,
\begin{enumerate}
\item[(a')]{$2g_i / (t_i m_i) \in \ZZ$}
\item[(b')]{$\deg(\alpha_i) = \Nrd_{R_i/\QQ}(\alpha_i)^{2g_i/(t_i m_i)}$.}
\end{enumerate}
To this end, we know that $\deg(\cdot) : R_i \rightarrow \QQ$ is a norm function of degree $2g_i r_i$ while $\Nrd_{R_i/\QQ}(\cdot)$ is a norm function of degree $t_i d_i$.  Thus by \cite[\S 19, Lem.~on p.~179]{Mum:v1}, we must have
\begin{equation}\label{3.15}
\deg(\cdot) = \Nrd_{R_i/\QQ}(\cdot)^{n_i}
\end{equation}
for some positive integer $n_i$.  In particular, $n_it_i d_i = 2g_i r_i$ and so 
\begin{equation}\label{3.15'}
n_i = 2g_i/(t_i m_i) \in \ZZ
\end{equation} which proves (a').  Assertion (b') follows from \eqref{3.15} and \eqref{3.15'}.
\end{proof}

If $B$ is an abelian variety then, by Poincar\'{e} reducibility, $B$ admits an isogeny $\beta : B \rightarrow A$ to an abelian variety $A$ of the form described in \S \ref{Sec2.5}, \cite[\S 19, Cor.~ 1, p.~174]{Mum:v1}; see \cite[Cor.~3.20]{Conrad:2006} or \cite[Thm.~1.2.1.3]{Chai:Conrad:Oort:2014} for the case that $\kk$ is not algebraically closed.  When we fix such an isogeny $\beta$, we can use Proposition \ref{theorem3.5} to understand the degree function as it pertains to elements of $\End^0(B)$. 

\begin{corollary}\label{corollary3.6}
If $A$ is as in \S \ref{Sec2.5} and if $\beta : B \rightarrow A$ is an isogeny, $\alpha \in \End^0(B)$ and $[\beta](\alpha) = (\alpha_1,\dots, \alpha_k) \in \End^0(A)$, then 
\begin{equation}\label{3.16}
\deg(\alpha) = \prod_{i=1}^k \Nrd_{R_i /\QQ}(\alpha_i)^{2g_i/(t_i m_i)}.
\end{equation}
\end{corollary}
\begin{proof}
Combine \eqref{3.14} and Proposition \ref{theorem3.5} (b).
\end{proof}

By combining Proposition \ref{theorem3.5} and Corollary \ref{corollary3.6} we obtain:

\begin{corollary}\label{corollary3.7}
Let $A$ be as in Proposition \ref{theorem3.5} and $B$ as in Corollary \ref{corollary3.6}.  Fix an ample divisor $\lambda$ on $A$.  The functions
\begin{equation}\label{3.17}
\prod_{i=1}^k \Nrd_{R_i/\QQ}(\cdot)^{2g_i/(t_im_i)}|_{\End^0_\lambda(A)} : \End^0_\lambda(A) \rightarrow \QQ 
\end{equation}
and
\begin{equation}\label{3.18}
\prod_{i=1}^k \Nrd_{R_i/\QQ}(\cdot)^{2g_i/(t_im_i)}|_{\End^0_{\beta^* \lambda}(B)} : \End^0_{\beta^*\lambda}(B) \rightarrow \QQ 
\end{equation}
are both squares of rational valued homogeneous polynomial functions of degree $g$ on $\End^0_\lambda(A)$ and $\End^0_{\beta^* \lambda}(B)$, respectively, normalized so as to take value $1$ when evaluated at $1_A$ and $1_B$, respectively.
\end{corollary}
\begin{proof}
The conclusions desired by Corollary \ref{corollary3.7} follow from Proposition \ref{theorem3.5} and \eqref{3.16} in conjunction with the Riemann-Roch theorem, \cite[\S 16, p.~150]{Mum:v1}.  Indeed, in order to establish \eqref{3.17}, let $\alpha = (\alpha_1,\dots, \alpha_k) \in \End^0_\lambda(A)$.  Then $\alpha = \Phi_\lambda(D)$ for some class of a $\QQ$-divisor $D$.  We first assume that $D$ is integral.  Then $\alpha = \phi^{-1}_\lambda \circ \phi_D$ and so, by \eqref{3.6},
\begin{equation}\label{3.19}
\deg(\alpha) = \deg(\phi^{-1}_\lambda) \deg(\phi_D).
\end{equation}
Using \eqref{3.19} together with Proposition \ref{theorem3.5} (b) and the Riemann-Roch theorem \cite[\S 16, p.~150]{Mum:v1}, we obtain
\begin{equation}\label{3.20}
\left( \frac{(D^g)}{g!}\right)^2 = \left( \frac{(\lambda^g)}{g!}\right)^2 \prod_{i=1}^k \Nrd_{R_i/\QQ}(\alpha_i)^{2g_i/(t_im_i)}.
\end{equation}
By $\QQ$-linearity, the formula \eqref{3.20} also holds for the case that $D$ is a $\QQ$-divisor.  Finally, since the left hand side of \eqref{3.20} is clearly the square of a homogeneous polynomial function of degree $g$ on $\End^0_\lambda(A)$, the same is true for the formula \eqref{3.17}.  

The formula \eqref{3.18} is established in the same manner by replacing $\lambda$ with $\beta^* \lambda$ and using the Riemann-Roch theorem for divisors on $B$.
\end{proof}

\np\label{} {\bf Notation.}  In what follows, we denote the normalized degree $g$ homogeneous polynomial functions whose squares are \eqref{3.17} and \eqref{3.18} respectively by
\begin{equation}\label{3.21}
\mathrm{pNrd}_\lambda(\cdot) : \End^0_\lambda(A) \rightarrow \QQ
\end{equation}
and
\begin{equation}\label{3.22}
\mathrm{pNrd}_{\beta^* \lambda}(\cdot) : \End^0_{\beta^* \lambda}(B) \rightarrow \QQ.
\end{equation}

\section{The Riemann-Roch and index theorems for Abelian varieties}\label{4}

In this section, we show how the Riemann-Roch theorem for divisors on an abelian variety is related to the reduced norm of the Wedderburn components of its $\QQ$-endomorphism algebra.  Specifically, in \S \ref{Sec4.1}, we prove Theorem \ref{theorem4.1} which has as a consequence Corollary \ref{corollary4.2}, equivalently Theorem \ref{theorem1.1} stated in \S \ref{1}.   We then show, in \S \ref{Sec4.2}, how Theorem \ref{theorem4.1} is related to the index theorem of Mumford \cite[\S 16, p.~155]{Mum:v1} which concerns understanding the cohomology groups of line bundles on abelian varieties.

\np\label{Sec4.1}{\bf The Riemann-Roch theorem.}  Suppose that $A = A_1^{r_1} \times \dots \times A_k^{r_k}$ with $A_i$ simple and pairwise nonisogenous abelian varieties.  Let $g := \dim A$, $g_i := \dim A_i$ and write $R := \End^0(A) = R_1 \times \dots \times R_k$ with $\Delta_i = \End^0(A_i)$ and $R_i = M_{r_i}(\Delta_i)$.  We also fix an ample divisor $\lambda$ on $A$, an isogeny $\beta : B \rightarrow A$, and we use the other notational conventions described in \S \ref{Sec2.2}.

We interpret the Riemann-Roch theorem, see \cite[\S 16, p.~150]{Mum:v1} for example, in terms of the polynomial functions \eqref{3.21} and \eqref{3.22} determined by the reduced norms of the Wedderburn components of $R$ and the numbers $m_i$, $t_i$ and $g_i$.

\begin{theorem}\label{theorem4.1}
In the setting of \S \ref{Sec4.1} just described, the following assertions hold true.
\begin{enumerate}
\item{If $D$ is a $\QQ$-divisor on $A$ and $\Phi_\lambda(D) = \alpha = (\alpha_1,\dots, \alpha_k) \in \End_\lambda^0(A)$, then
\begin{equation}\label{4.1}
\frac{(D^g)}{g!} = \frac{(\lambda^g)}{g!} \mathrm{pNrd}_\lambda(\alpha) = \sqrt{\deg \phi_\lambda}\mathrm{pNrd}_\lambda(\alpha).
\end{equation} }
\item{If $D$ is a $\QQ$-divisor on $B$ and 
$$[\beta]\Phi_{\beta^* \lambda}(D) = \alpha = (\alpha_1,\dots, \alpha_k) \in [\beta]\End^0_{\beta^* \lambda}(B) \subset \End^0_\lambda(A), $$
then
\begin{multline}\label{4.2}
\frac{(D^g)}{g!} = \frac{(\beta^* \lambda^g)}{g!} \mathrm{pNrd}_{\beta^*\lambda}(\Phi_{\beta^* \lambda}(D))  
=  
\frac{(\beta^* \lambda^g)}{g!} \mathrm{pNrd}_{\lambda}(\alpha) = \sqrt{\deg \phi_{\beta^* \lambda}}\mathrm{pNrd}_\lambda(\alpha).
\end{multline}
}
\end{enumerate}
\end{theorem}
\begin{proof}
We content ourselves with proving \eqref{4.1} as the proof of \eqref{4.2} follows similarly by replacing $\lambda$ with $\beta^*\lambda$.  To achieve our goal, we have, by \eqref{3.20} and \eqref{3.21}, that 
\begin{equation}\label{4.3}
\frac{(D^g)}{g!} = \mathrm{C} \times \frac{(\lambda^g)}{g!} \mathrm{pNrd}_\lambda(\alpha),
\end{equation}
for some constant $\mathrm{C}$ independent of $D$.  To solve for $\mathrm{C}$, we replace $D$ by $\lambda$ in \eqref{4.3}.  Since $\Phi_\lambda(\lambda) = 1_R$ and $\mathrm{pNrd}_\lambda(1_R) = 1$, we conclude that $\mathrm{C}=1$.
\end{proof}

As a special case of Theorem \ref{theorem4.1} we have the following consequence which we also stated in \S \ref{1} as Theorem \ref{theorem1.1}.
 
\begin{corollary}\label{corollary4.2}
If $\lambda$ is a principal polarization on $A$, $D$ a divisor on $A$ and $\alpha = \Phi_\lambda(D)$, then
$$
\frac{(D^g)}{g!} = \mathrm{pNrd}_\lambda(\alpha).
$$
\end{corollary}
\begin{proof}[Proof of Corollary \ref{corollary4.2} and Theorem \ref{theorem1.1}]
Follows from \eqref{4.1} because if $\lambda$ is a principal polarization, then $(\lambda^g) = g!$.
\end{proof}

\np\label{Sec4.2} {\bf The index theorem.}  Let $A$ be an abelian variety of dimension $g$.  We let $\lambda$ be an ample divisor on $A$ and $D$ a $\QQ$-divisor on $A$.  The \emph{Hilbert polynomial} of $D$ with respect to $\lambda$ is the polynomial
\begin{equation}\label{4.5}
\mathrm{Hp}_\lambda(D;N) := \frac{1}{g!}((N\lambda+D)^g) \in \QQ[N].
\end{equation}

In the case that $D$ is integral, then the complex roots of the polynomial \eqref{4.5} are real \cite[Thm.~2, p.~98]{Mum:Quad:Eqns} and \cite[\S 16, p.~155]{Mum:v1}.  Furthermore, the dimension of the scheme theoretic kernel of $\phi_D$ equals the multiplicity of zero as a root of \eqref{4.5}.  Also, counting the (real) roots of $\mathrm{Hp}_\lambda(D;N)$ with multiplicity, it follows that $\H^j(A,\Osh_A(D)) = 0$ for $0 \leq j < \# \text{ of positive roots}$ and $\H^{g-j}(A,\Osh_A(D)) = 0$ for $0 \leq j < \# \text{ of negative roots}$, \cite[Thm.~2, p.~98]{Mum:Quad:Eqns}.  (Note that when $\kk$ is not algebraically closed we reduce to that case by first applying the flat base change theorem, compare with \cite[Lem.~2.4, p.~1727]{Shin}.)

Let $\ii(D)$ be the number of positive roots of the polynomial \eqref{4.5}.  This number is independent of the choice of ample divisor $\lambda$, \cite[p.~99]{Mum:Quad:Eqns}, and we refer to it as the \emph{index} of $D$.  Since numerically equivalent divisors have the same Hilbert polynomial, the polynomial \eqref{4.5} is well-defined on numerical equivalence classes of divisors.  As a consequence, the number $\ii(D)$ is also well-defined modulo numerical equivalence.  In particular it is well-defined modulo algebraic equivalence.

To define the index of a $\QQ$-divisor $D$, let $n $ be a positive integer so that $nD$ is integral and define $\ii(D) = \ii(nD)$.

\begin{lemma}\label{lemma3.3}
Let $D$ be a $\QQ$-divisor on  $A$.  The number $\ii(D)$ is well-defined and equals the number of positive roots counted with multiplicity of the polynomial $\mathrm{Hp}_\lambda(D;N)$.
\end{lemma}
\begin{proof}
We simply note
$$
\mathrm{Hp}_{n\lambda}(nD;N) = \frac{1}{g!} ((Nn\lambda+ nD)^g) = \frac{n^g}{g!}((N\lambda + D)^g) = n^g \mathrm{Hp}_{\lambda}(D;N).
$$
\end{proof}

We now return to the setting of \S \ref{4.1}.  The following theorem and its corollary are due to Mumford. Indeed, they can be seen as slightly more explicit forms of calculations performed in \cite[\S 21, p.~209]{Mum:v1}.

\begin{theorem}\label{theorem4.4}
In the setting of \S \ref{4.1}, the following assertions hold true.
\begin{enumerate}
\item{If $D$ is a $\QQ$-divisor on $A$ and $\Phi_\lambda(D) = \alpha = (\alpha_1,\dots, \alpha_k)$, the image of $D$ in $\End_\lambda^0(A)$, then $\ii(D)$ equals the number of positive roots counted with multiplicity of the polynomial
\begin{equation}\label{4.7}
p_{D,\lambda}(N) := \mathrm{pNrd}_\lambda(N\operatorname{id}_A + \alpha).
\end{equation}
}
\item{
If $D$ is a $\QQ$-divisor on $B$ and $[\beta]\Phi_{\beta^* \lambda}(D) = \alpha \in \End^0_{\lambda}(A)$, then $\ii(D)$ equals the number of positive roots counted with multiplicity of the polynomial
\begin{equation}\label{4.8}
p_{D,\beta^* \lambda}(N) := \mathrm{pNrd}_{\beta^*\lambda}(N \operatorname{id}_B + \Phi_{\beta^* \lambda}(D)) = \mathrm{pNrd}_\lambda(N\operatorname{id}_A + \alpha).
\end{equation}
}
\end{enumerate}
\end{theorem}
\begin{proof}
To prove (a), by Lemma \ref{4.3}, we need to determine the number of positive roots of the polynomial
$$
\mathrm{Hp}_\lambda(D;N) = \frac{1}{g!}((N\lambda+D)^g).
$$
On the other hand, by Theorem \ref{theorem4.1} (a), the roots of this polynomial are the same as those of the polynomial
$$
\frac{(\lambda^g)}{g!}\mathrm{pNrd}_\lambda(N\operatorname{id}_A + \alpha)
$$
and (a) clearly follows.   Assertion (b) follows similarly with $\lambda$ replaced by $\beta^* \lambda$.
\end{proof}

\section{The index theorem for simple abelian varieties}\label{5}

In this section we consider consequences of Theorem \ref{theorem4.4} when applied to simple abelian varieties.  For convenience of the reader and to help keep some of the following discussion self contained, we start by recalling Albert's classification of division rings with positive involution in \S \ref{5.0}.  We then explain, in \S \ref{5.0'}, how this classification relates to the endomorphism algebras of simple  abelian varieties.  Having recalled some prerequisites in \S \ref{5.0}, \S \ref{5.0'} and \S \ref{5.0'.1}, in \S \ref{Sec5.1} we consider some consequences of Theorem \ref{theorem4.4} in some concrete special cases.  Finally, in \S \ref{5.5} we prove Theorem \ref{theorem5.1} which is motivated by our previous works \cite{Grieve-cup-prod-ab-var} and \cite{Grieve:theta}.

\np\label{5.0}{\bf Division rings with positive involution.}
Albert classified pairs $(\Delta, ')$ where $\Delta$ is a division ring of finite dimension $n := \dim_{\QQ} \Delta$ and $' : \Delta \rightarrow \Delta$ is a positive involution.
We recall here, following the presentation of Mumford \cite[\S 21, p.~193-203]{Mum:v1} closely, some aspects of this classification.  To do so, the centre of a division ring $\Delta$ will be denoted by $Z$ and we let $K := \{\alpha \in Z : \alpha' = \alpha \}$ be the set of elements of $Z$ fixed by the involution $' : \Delta \rightarrow \Delta$.
We also put $e := [K:\QQ]$, $t:=[Z:\QQ]$ and $m^2 := \dim_Z \Delta$.  We then have $n = t m^2$.

Before describing Albert's classification, we fix some terminology and notation. First, by a \emph{quaternion division algebra} over a totally real algebraic number field $K$ we mean a central division algebra $\Delta$ of dimension $4$ over $K$. Second, if $\Delta$ is a central division ring over a  number field $K$, $\nu$ a place of $K$ and $K_\nu$ the completion of $K$ with respect to $\nu$, then, as in \cite[\S 21, Thm., p.~196]{Mum:v1}, we denote by $\mathrm{Inv}_\nu(\Delta)$ the element of $\QQ/\ZZ$ corresponding to the class of $K_\nu \otimes_K \Delta$ in the Brauer group $\mathrm{Br}(K_\nu)$.  Finally if $K$ is a totally real number field and $\sigma : K \hookrightarrow \RR$ an embedding, then we denote by $\RR_{(\sigma)}$ the real numbers $\RR$ considered as a $K$-algebra via $\sigma$.

Having made these remarks, Albert's classification is as follows.

\begin{theorem}[{\cite[\S 21, Thm.~2, p.~201]{Mum:v1}}, {\cite[Prop.~1]{Shimura}}]
\label{Albert:Classified}
Let $\Delta$ be a division algebra of finite dimension over $\QQ$ with positive involution $'$.  Let $Z$ be the centre of $\Delta$ and $K$ the subfield of elements of $Z$ fixed by $'$.  Then $(\Delta,')$ is one of the following four types.

\noindent
Type I.  $\Delta = Z = K$ is a totally real algebraic number field and the involution $'$ is the identity
$$
\alpha'=\alpha.
$$

\noindent
Type II.  $Z = K$ is a totally real algebraic number field and $\Delta$ is a quaternion division algebra over $K$ such that for every embedding $\sigma : K \hookrightarrow \RR$ we have $\RR_{(\sigma)} \otimes_K \Delta \xrightarrow{\sim} M_2(\RR). $
Furthermore, the involution $' : \Delta \rightarrow \Delta$ has the form 
\begin{equation}\label{typeII:involution}
 \alpha' = \gamma (\Trd_{\Delta / K}(\alpha) - \alpha) \gamma^{-1},
\end{equation}
for some $\gamma \in \Delta$ with $\gamma^2 \in K$ and $\gamma^2$ totally negative. Conversely, every involution of the form \eqref{typeII:involution} defines a positive involution on $\Delta$.

\noindent
Type III.  $Z = K$ is a totally real algebraic number field and $\Delta$ is a quaternion division algebra over $K$ such that for every embedding $\sigma: K \hookrightarrow \RR$, $\RR_{(\sigma)} \otimes_K \Delta \simeq \mathbb{H}$ for $\mathbb{H}$ the algebra of Hamiltonian quaternions on $\RR$.  In this case, the involution $'$ takes the form 
$$
\alpha' = \Trd_{\Delta / K}(\alpha) - \alpha.
$$

\noindent
Type IV.  $K$ is a totally real algebraic number field and $Z$ is a totally imaginary quadratic extension of $K$ with conjugation $\tau$ over $Z$.  Then $\Delta$ is a division algebra with centre $Z$ so that the following hold true.
\begin{enumerate}
\item{If $\nu$ is a finite place of $Z$ fixed by $\tau$, then $\mathrm{Inv}_\nu(\Delta) = 0$.}
\item{For every finite place $\nu$ of $Z$, $\mathrm{Inv}_\nu(\Delta) + \mathrm{Inv}_{\tau \nu}(\Delta) = 0$.}
\end{enumerate}
In this case, there exists a positive involution $* : \Delta \rightarrow \Delta$ together with an isomorphism
\begin{equation}\label{ab:eqn1}
\RR \otimes_{\QQ} \Delta \xrightarrow{\sim} M_d(\CC)\times \dots \times M_d(\CC)
\end{equation}
which carries the involution $*$ to the involution
$$(X_1,\dots, X_e) \mapsto (\overline{X}_1^\mathrm{T},\dots, \overline{X}_e^\mathrm{T}), $$
for $\overline{X}_i^{\mathrm{T}}$ the conjugate transpose of $X_i$.

Furthermore, there exists $\gamma \in \Delta$ with $\gamma^* = \gamma$ so that the image of $1\otimes \gamma$ under \eqref{ab:eqn1} is of the form $(\mathrm{A}_1,\dots,\mathrm{A}_e)$ with $\mathrm{A}_i$ Hermitian positive definite matrices and so that the given involution $' : \Delta \rightarrow \Delta$ has the form 
\begin{equation}\label{TypeIV:involution}
\alpha'=\gamma\alpha^* \gamma^{-1}.
\end{equation}
Conversely, every positive involution on $\Delta$ has the form \eqref{TypeIV:involution} for some $\gamma \in \Delta$ with $\gamma^* = \gamma$.
\end{theorem}
\begin{proof}
See \cite[p.~193-209]{Mum:v1}.
\end{proof}

\np\label{Geometric:Albert:table}{\bf Albert's classification and simple abelian varieties.}\label{5.0'}  We now consider Albert's classification, Theorem \ref{Albert:Classified}, as it pertains to simple abelian varieties.  Specifically, for a simple abelian variety $A$, $\Delta := \End^0(A)$ is a division ring of finite dimension over $\QQ$ and every ample divisor $\lambda$ on $A$ determines a positive involution $r_\lambda : \Delta \rightarrow \Delta$.  The pair $(\Delta, r_\lambda)$ belongs to one of the four types described in Theorem \ref{Albert:Classified} but, as explained in \cite[\S 21, p.~202]{Mum:v1}, the geometry of $A$ places further restrictions on some of the numerical invariants that we can associate to the pair $(\Delta, r_\lambda)$.  These restrictions are summarized in the table 
below which is only notationally different from that of \cite[\S 21, p.~202]{Mum:v1} and \cite[\S 1.3.6.3]{Chai:Conrad:Oort:2014}.

\begin{table}[h]
\label{AlbertTable}
\begin{center}
\begin{tabular}{| l | l | l | l | p{3cm} | p{3cm} |}  \noalign{} \hline \noalign{}
Type & $t$ & $m$ &  $\frac{\dim_\QQ \End^0_\lambda(A)}{\dim_\QQ \End^0(A)}$ & Restriction when $\cchar \kk = 0$, $\dim A = g$ & Restriction when $\cchar \kk > 0$, $\dim A = g$ \\ \noalign{} \hline \noalign{}
I & $e$ & $1$ & $1$ & $t | g$ & $t | g$  \\ \noalign{} \hline \noalign{}
II & $e$ & $2$ & $3/4$ & $2t | g$ & $2t | g $\\  \noalign{} \hline \noalign{}
III & $e$ & $2$ & $1/4$ & $2t | g$ & $t | g$ \\ \noalign{} \hline \noalign{}
IV & $2e $& $m$ & $1/2 $& $em^2 | g$ & $em | g$ \\ \noalign{} \hline \noalign{}
\end{tabular}
\end{center} 
\end{table}

\np\label{5.0'.1} In light of \S \ref{5.0} and \S \ref{5.0'}, it is an interesting question to decide given a division algebra  with positive involution $(\Delta,')$, of one of the types described in Theorem \ref{Albert:Classified} and satisfying the numerical constraints given in \S \ref{5.0'}, if there exists a simple abelian variety $A$ together with an ample divisor $\lambda$ so that the pair $(\End^0(A),r_\lambda)$ is isomorphic to $(\Delta,')$.  As explained in \cite[\S 21, p.~203]{Mum:v1}, when $\cchar \kk =  0$ it is classically known that such a pair $(A,\lambda)$ exists except in the case that $(\Delta, ')$ is of Type III and $g/(2t)$ equals $1$ or $2$ or $(\Delta,')$ is of Type IV and $g/(em^2)$ equals $1$ or $2$.  In these exceptional cases it is known what further restrictions ensure the existence of such an $(A,r_\lambda)$, see \cite[\S 4]{Shimura}.  For the situation that $\cchar \kk > 0$, we refer to \cite{Oort:1988} and \cite{Oort:VanDerPut:1988} for more details, especially \cite[\S 8]{Oort:1988}.

\np\label{Sec5.1}\label{5.0''}  We now consider some aspects of \S \ref{Sec4.2} in the context of \S \ref{5.0} and \S \ref{5.0'}.  To do so, let $A$ be a simple abelian variety of dimension $g$, $\lambda$ an ample divisor on $A$, $\Delta := \End^0(A)$, $Z$ the centre of $\Delta$ and $K \subseteq Z$ the subfield fixed by the Rosati involution $r_\lambda$.  Then $K$, by Theorem \ref{Albert:Classified}, is a totally real field.  As in \S \ref{Sec2.3} we let $m$ denote the Schur index of the division ring $\Delta$, $t := [Z:\QQ]$,  $e := [K:\QQ]$ and $\sigma_1,\dots,\sigma_e$ the embeddings of $K$ into $\RR$.

Suppose that $D$ is a divisor on $A$ with the property that $\Phi_\lambda(D) = \alpha \in K$.  We then have
$$
\Nrd_{\Delta/\QQ}(\alpha) = \Nr_{K/\QQ}(\alpha)^{m[Z:K]} = \prod_{i=1}^e \sigma_i(\alpha)^{m[Z:K]},
$$
and so the polynomial 
$$p_{D,\lambda}(N) = \mathrm{pNrd}_\lambda(N\operatorname{id}_A + \alpha)$$
takes the form
\begin{equation}\label{5.2}
p_{D,\lambda}(N) = \mathrm{pNrd}_\lambda(N\operatorname{id}_A + \alpha) = \prod_{i=1}^e (N\operatorname{id}_A + \sigma_i(\alpha))^{g/e},
\end{equation}
where by $\prod_{i=1}^e(N\operatorname{id}_A + \sigma_i(\alpha))^{g/e}$ in \eqref{5.2} we mean the normalized polynomial function from $\ZZ \times K$ to $\QQ$ whose square is $\prod_{i=1}^e(N\operatorname{id}_A + \sigma_i(\alpha))^{2g/e}$.

In what follows if $\alpha \in K$, then we let 
\begin{equation}\label{inversion:set} \Sigma^-(\alpha) := \{\sigma_j : \sigma_j(\alpha) < 0 \}
\end{equation}
denote the set of embeddings $\sigma_j$ of $K$ for which $\sigma_j(\alpha) < 0$.

\noindent
{\bf Examples.}  
\begin{enumerate}
\item[(i)]{If the pair $(\Delta,r_\lambda)$ is Type I in the sense of Theorem \ref{Albert:Classified}, 
then $\Delta = K$ and $e | g$, by \S \ref{Geometric:Albert:table}.  The polynomial \eqref{5.2} then takes the form
\begin{equation}\label{5.3}
p_{D,\lambda}(N) = \mathrm{pNrd}_\lambda(N \operatorname{id}_A + \alpha) = \Nr_{K/\QQ}(N \operatorname{id}_A + \alpha)^{g/e}.
\end{equation} 
}
\item[(ii)]{Suppose that the pair $(\Delta,r_\lambda)$ is Type IV in the sense of Theorem \ref{Albert:Classified}
but with $\Delta=Z$.  Then $\Delta$ is a totally imaginary quadratic extension of $K$ and $e | g$,  in light of \S \ref{Geometric:Albert:table}.
As a consequence, the polynomial \eqref{5.2} then takes the form
\begin{equation}\label{5.4}
p_{D,\lambda}(N) = \mathrm{pNrd}_\lambda(N\operatorname{id}_A + \alpha) = \Nr_{K/\QQ}(N \operatorname{id}_A + \alpha)^{g/e}.
\end{equation}
}
\end{enumerate}

\np\label{5.5} Next, we consider Theorem \ref{theorem4.4} as well as the discussion of \S\S \ref{5.0}-\ref{Sec5.1} in the context of our previous works \cite{Grieve-cup-prod-ab-var} and \cite{Grieve:theta}.  Specifically, we prove Theorem \ref{theorem5.1} which, among other things, shows how these matters are related to the \emph{pair index condition}, a concept we introduced in \cite{Grieve-cup-prod-ab-var}.
To motivate Theorem \ref{theorem5.1} and also to help further motivate some of what we do in the present article, we start by summarizing some of our \cite{Grieve-cup-prod-ab-var}. 

In more detail, motivated in part by the index theorem discussed in  \S \ref{4.2}, as well as the study of cup-product problems of the form
\begin{equation}\label{cup:prod:1} \H^i(A,\Osh_A(D_1)) \times \H^j(A,\Osh_A(D_2)) \xrightarrow{\bigcup} \H^{i+j}(A,\Osh_A(D_1 +D_2)), 
\end{equation}
determined by pairs of divisors $(D_1,D_2)$ on an abelian variety $A$, in \cite{Grieve-cup-prod-ab-var} we defined the \emph{pair index condition} for a pair of divisors $(D_1,D_2)$ on $A$ to be the condition that the Euler characteristics of $\Osh_A(D_1)$, $\Osh_A(D_2)$, and $\Osh_A(D_1+D_2)$ are nonzero and  $\ii(D_1) + \ii(D_2) = \ii(D_1 + D_2)$. 

The point is that if the Euler characteristics of $\Osh_A(D_1)$, $\Osh_A(D_2)$, and $\Osh_A(D_1+D_2)$ are nonzero, then the condition that $i = \ii(D_1)$, $j = \ii(D_2)$ and $\ii(D_1) + \ii(D_2) = \ii(D_1 + D_2)$ is necessary for the map \eqref{cup:prod:1} to be non-zero; as shown in \cite[\S 7.1.3]{Grieve-cup-prod-ab-var}, this condition is not sufficient in general.

In \cite[\S 2]{Grieve-cup-prod-ab-var} we studied the pair index condition as it pertains to complex abelian varieties with real multiplication.  For example, in \cite[Thm.~2.2, p.~1450036-8]{Grieve-cup-prod-ab-var}, we showed that this condition can be satisfied in all possible instances by certain classes of simple complex abelian varieties with real multiplication by a totally real number field of degree $g$ over $\QQ$.  We also established the more general \cite[Thm.~2.3, p.~1450036-9]{Grieve-cup-prod-ab-var} which concerns satisfying the pair index condition as it applies to certain simple semihomogenous vector bundles on complex abelian varieties with real multiplication.

Here, motivated by these results of \cite{Grieve-cup-prod-ab-var}, we use Theorem \ref{theorem4.4} to prove Theorem \ref{theorem5.1} which gives a fairly complete characterization of the pair index condition as it pertains to pairs of line bundles on simple abelian varieties.

\begin{theorem}\label{theorem5.1}
In the setting of \S \ref{Sec5.1}, and using the notation \eqref{inversion:set}, the following assertions hold true.
\begin{enumerate}
\item{If $D \in \NS^0(A)$ and $\Phi_\lambda(D) = \alpha \in K$, then $\ii(D) = \frac{g}{e} \# \Sigma^-(\alpha)$.}
\item{If $D_1,D_2 \in \NS^0(A)$ are such that $\Phi_\lambda(D_1) = \alpha_1 \in K$ and $\Phi_\lambda(D_2) = \alpha_2 \in K$, then 
$$\ii(D_1) + \ii(D_2) = \ii(D_1 + D_2)$$ if and only if
$$\Sigma^-(\alpha_1 + \alpha_2)  = \Sigma^-(\alpha_1)  \sqcup \Sigma^-(\alpha_2).$$}
\item{Let $n = g/e$.  Then for all $p,q \geq 0$ with $nq + np \leq g$, there exist integral classes $D_1,D_2 \in \NS^0(A)$ which satisfy the conditions that
\begin{enumerate}
\item[{\rm(i)}]{$(D_1^g) \not = 0$, $(D_2^g) \not = 0$; and}
\item[{\rm(ii)}]{$\Phi_\lambda(D_1) \in K$, $\Phi_\lambda(D_2) \in K$; and}
\item[{\rm(iii)}]{$\ii(D_1) = np$, $\ii(D_2) = nq$, and $\ii(D_1 + D_2) = \ii(D_1) + \ii(D_2)$.}
\end{enumerate}}
\end{enumerate}
\end{theorem}
\begin{proof}
Part (a) is a consequence of \eqref{5.2} combined with  Theorem \ref{theorem4.4} (a).  Part (b) is an immediate consequence of (a).  Part (c) can be proven in a manner similar to the proof of \cite[Thm.~2.2]{Grieve-cup-prod-ab-var}.  Indeed, first fix disjoint subsets $I, J \subseteq \{1,\dots,e \}$ with $\# I = p$ and $\# J = q$.  Next, let $U \subseteq \RR^{\oplus 2e}$ be the subset consisting of those $(x_1,\dots, x_e,y_1,\dots,y_e) \in \RR^{\oplus 2 e}$ with the property that $x_k < 0$ if $k \in I$, $x_k>0$ if $k \not \in I$, $y_k<0$ if $k \in J$, $y_k>0$ if $k \not \in J$ and $x_k + y_k < 0$ for all $k \in I \cup J$.  Then $U$ is a non-empty open subset of $\RR^{\oplus 2e}$.  On the other hand, by \cite[p.~135]{Mar} for instance, the set
$$ S := \{(\sigma_1(\alpha_1),\dots,\sigma_e(\alpha_1), \sigma_1(\alpha_2),\dots, \sigma_e(\alpha_2)) : (\alpha_1,\alpha_2) \in K^{\oplus 2} \}$$
is dense in $\RR^{\oplus 2e}$ and we deduce that $U \cap S \not = \emptyset$. 

Considering the definitions of $U$ and $S$, we conclude that there exist $\alpha_1,\alpha_2 \in K$ with the property that
\begin{enumerate}
\item[(a')]{$\sigma_k(\alpha_1) < 0$ if $k \in I$ and $\sigma_k(\alpha_1) > 0$ if $k \not \in I$}
\item[(b')]{$\sigma_k(\alpha_2) < 0$ if $k \in J$ and $\sigma_k(\alpha_2) > 0$ if $k \not \in J$}
\item[(c')]{$\sigma_k(\alpha_1) + \sigma_k(\alpha_2) < 0$ if $k \in I \cup J$.}
\end{enumerate}
Consequently, if we fix $D_1,D_2\in \NS^0(A)$ with $\rho_{\lambda}(D_1) = \alpha_1$ and $\rho_{\lambda}(D_2) = \alpha_2$, then some positive scalar multiple of $D_1$ and $D_2$ will satisfy the conditions desired by (c).
\end{proof}

\begin{corollary}
For each $p$, $0\leq p \leq e$, $A$ admits an integral divisor $D$ with $\ii(D) = gp/e$.
\end{corollary}
\begin{proof}
Follows from Theorem \ref{theorem5.1} (c).
\end{proof}

\section{Products of abelian varieties}\label{6}

In order to describe the right hand side of \eqref{4.7} and \eqref{4.8} explicitly, we need a clear understanding of $\End^0_\lambda(A)$.  To get a sense for some of the issues involved, in this section we consider the case of product polarizations on products of abelian varieties.  The main point is to prove Proposition \ref{6.1} and Corollary \ref{6.2} which we use in \S \ref{7} where we consider examples.

\np\label{} Let $A_i$, for $i=1,\dots,k$, be abelian varieties with ample divisors $\lambda_i$.  As in \cite[p.~50]{Kani:2016},  for each $\eta_{ji} \in \Hom^0(A_i,A_j)$, let
\begin{equation}\label{6.1}
r_{\lambda_i,\lambda_j}(\eta_{ji}) = \phi^{-1}_{\lambda_i} \circ \widehat{\eta_{ji}} \circ \phi_{\lambda_j} \in \Hom^0(A_j,A_i).
\end{equation}
If each of the $\lambda_i$ determine principal polarizations of the $A_i$, then $r_{\lambda_i,\lambda_j}(\eta_{ji}) \in \Hom(A_j,A_i)$ whenever $\eta_{ji} \in \Hom(A_i,A_j)$.

\np\label{Sec6.2} Let $A = A_1\times \dots \times A_k$, $p_i : A \rightarrow A_i$ the projection and $e_i : A_i \hookrightarrow A$ the inclusion.  Let $\lambda$ denote the ample divisor $\sum_{i=1}^k p_i^* \lambda_i$ on $A$.
If $\alpha \in \End^0(A)$, then define $\alpha_{ij} := p_i \alpha e_j \in \Hom^0(A_j,A_i)$.
We then have:
\begin{equation}\label{6.2}
\End(A) = \{ (\alpha_{ij})_{1\leq i,j \leq k} : \alpha_{ij} \in \Hom(A_j,A_i) \}
\end{equation}
and
\begin{equation}\label{6.3}
\End^0(A) = \{ (\alpha_{ij})_{1\leq i,j \leq k} : \alpha_{ij} \in \Hom^0(A_j,A_i)\}.
\end{equation}

\np\label{} The following proposition is useful for working with examples and is an evident extension of \cite[Prop.~61, p.~51]{Kani:2016}; see also \cite[Lem.~2.3, p.~88]{Katsura:2015} and \cite[p.~1003]{Bauer:1998}.

\begin{proposition}\label{proposition6.1}
In the setting of \S \ref{Sec6.2}, we have that
\begin{equation}\label{6.4}
\End^0_\lambda(A) = \{ (\alpha_{ij}) : \alpha_{ij} = r_{\lambda_i,\lambda_j}(\alpha_{ji}), \alpha_{ij} \in \Hom^0(A_j,A_i)\}.
\end{equation}
Also, if each of the $\lambda_i$ determine principal polarizations of the $A_i$, then it is also true that
\begin{equation}\label{6.5}
\End_\lambda(B) = \{ (\alpha_{ij}) : \alpha_{ij}=r_{\lambda_i,\lambda_j}(\alpha_{ji}), \alpha_{ij} \in \Hom(A_j,A_i)\}.
\end{equation}
\end{proposition}
\begin{proof}
See \cite[p.~51]{Kani:2016} for the case $k=2$ and $\lambda_i$ principal polarizations.   The desired more general case is proved similarly.
\end{proof}

A consequence of Proposition \ref{proposition6.1} is:

\begin{corollary}\label{corollary6.2}
If $B = A^r = \underbrace{A \times \dots \times A}_{r \text{-times}}$ and $\lambda^{\boxtimes r} = \sum_{i=1}^r p_i^* \lambda$, for $\lambda$ an ample divisor on $A$, then
\begin{equation}\label{6.6}
\End^0_{\lambda^{\boxtimes r}}(B) = \{ (\alpha_{ij}) : \alpha_{ij} = r_\lambda(\alpha_{ji}) , \alpha_{ij} \in \End^0(A)\}.
\end{equation}
If $\lambda$ determines a principal polarization, then 
\begin{equation}\label{6.7}
\End_{\lambda^{\boxtimes r}}(B) = \{ (\alpha_{ij}) : \alpha_{ij} = r_\lambda(\alpha_{ji}), \alpha_{ij} \in \End(A)\}.
\end{equation}
\end{corollary}
\begin{proof}
The conclusion desired by Corollary \ref{corollary6.2} follows from Proposition \ref{6.1} because $\lambda_i = \lambda$, for $i=1,\dots, k$, and so $r_{\lambda_i,\lambda_j}(\alpha_{ji}) = r_\lambda(\alpha_{ji})$.
\end{proof}

\np\label{}{\bf Remark.}  
As in \cite[Prop.~61]{Kani:2016}, the equality \eqref{6.4}, and similarly for \eqref{6.5}, can be seen as an isomorphism
\begin{equation}\label{6.8}
\mathrm{D} = \mathrm{D}_{\lambda_1,\dots, \lambda_k} : \bigoplus_i \End^0_{\lambda_i}(A_i) \oplus \bigoplus_{i<j} \Hom^0(A_i,A_j) \xrightarrow{\sim} \End^0_\lambda(A)
\end{equation}
which induces an isomorphism
\begin{equation}\label{6.9}
\mathrm{D} = \mathrm{D}_{\lambda_1,\dots,\lambda_k} : \bigoplus_i \NS^0(A_i) \oplus \bigoplus_{i<j} \Hom^0(A_i,A_j) \xrightarrow{\sim} \NS^0(A).
\end{equation}
When each of the $\lambda_i$ are principal polarizations, we similarly have isomorphisms
\begin{equation}\label{6.10}
\mathrm{D} = \mathrm{D}_{\lambda_1,\dots, \lambda_k} : \bigoplus_i \End_{\lambda_i}(A_i) \oplus \bigoplus_{i<j} \Hom(A_i,A_j) \xrightarrow{\sim} \End_\lambda(A)
\end{equation}
and
\begin{equation}\label{6.11}
\mathrm{D} = \mathrm{D}_{\lambda_1,\dots,\lambda_k} : \bigoplus_i \NS(A_i) \oplus \bigoplus_{i<j} \Hom(A_i,A_j) \xrightarrow{\sim} \NS(A).
\end{equation}

When $A = E_1 \times E_2$, for $E_i$ elliptic curves, the isomorphisms \eqref{6.10} and \eqref{6.11} are important in establishing the main results of \cite{Kani:2016}.  Here the main focus is not the isomorphisms \eqref{6.8}--\eqref{6.11} but rather the useful \eqref{6.6} which we use to illustrate some of our theorems here, for instance Theorems \ref{theorem4.1} and \ref{theorem4.4}.

\section{Examples}\label{7}
In this section we show how Theorem \ref{theorem4.1}, Corollary \ref{corollary4.2} and Theorem \ref{theorem4.4}  can be used to study the N\'{e}ron-Severi space of the self product of certain simple abelian varieties.  Our calculations, among other things, allow for a more complete understanding of \cite[E.g.~7.1.1, p.~1450036-25]{Grieve-cup-prod-ab-var}.

\np\label{Sec7.1} Let us suppose that $B = A^r$ for $(A,\lambda)$ a principally polarized abelian variety with $\End(A)$ an order $\mathfrak{o}_K$ in $K$, a totally real number field of degree $g$ over $\QQ$.
We then have that 
$\End_\lambda(A) = \End(A) = \mathfrak{o}_K,$ \cite[\S 21, p.~201]{Mum:v1},  
$\End(B) = M_r(\mathfrak{o}_K)$ and $\End^0(B) = M_r(K).$ 

Also, by Corollary \ref{corollary6.2}, it follows that
$$
\End^0_{\lambda^{\boxtimes r}}(B) = \{ (\alpha_{ij}) \in M_r(K) : (\alpha_{ij}) = (\alpha_{ji})\},
$$
the space of symmetric matrices with entries in $K$.  In addition, if $\alpha = (\alpha_{ij}) \in \End^0_\lambda(B)$, then
\begin{equation}\label{7.2}
\deg(\alpha) = \Nr_{K/\QQ}(\det (\alpha));
\end{equation}
here, in \eqref{7.2}, we view $\alpha$ as a $K$-linear operator and so its determinant is relative to $K$.

Note also that if $\alpha = (\alpha_{ij}) \in \End^0_\lambda(B)$ has the form
$\alpha = \Phi_{\lambda^{\boxtimes r}}(D)$ for some $D \in \NS^0(A)$, then
$$
\frac{(D^g)}{g!}= \Nr_{K/\QQ} (\det(\alpha))
$$
and the polynomial $p_{D,\lambda}(N)$ has the form
$$
p_{D,\lambda}(N) = \mathrm{pNrd}_\lambda(N \operatorname{id}_A + \alpha)=\Nr_{K/\QQ} \det(N\operatorname{id}_A + \alpha).
$$

\np\label{Sec7.2}  Let us now restrict the discussion of \S \ref{Sec7.1} to the case that  $g=1$ and $r=2$.  In this case we have that $A = E$, an elliptic curve with $\End(E) = \ZZ$.  Let $\lambda = 0_E$ the identity of $E$.  We then have
$$ \End^0_{\lambda^{\boxtimes 2}}(B) = \left \{ \left( \begin{matrix}  x & y \\ y & w \end{matrix} \right) : \text{ $x,y,w \in \QQ$} \right\},$$
and if $D \in \NS^0(A)$ has image 
$$\Phi_\lambda(D) = \left( \begin{matrix} 
x & y \\ y & w 
\end{matrix}  \right) $$
in $\End^0_{\lambda^{\boxtimes 2}}(B)$, then 
\begin{equation}\label{7.5}
p_{D,\lambda}(N) = \det \left( \begin{matrix}
N+x & y \\ y & N + w
\end{matrix} \right) = N^2 + (x+w)N + wx - y^2.
\end{equation}

Considering the roots of the polynomial \eqref{7.5} we deduce that $N = 0$ is a root of multiplicity $2$ if and only if $x=y=w=0$.  On the other hand, if $N=0$ is a root of multiplicity $1$, then $\ii(D) = 0$ if and only if $x+w > 0$ and $\ii(D) = 1$ if and only if $x+w < 0$.  Similarly, $N = 0$ is not a root if and only if $wx-y^2 \not = 0$ and in this case $\ii(D) = 0$ if and only if $wx-y^2>0$ and $x+w > 0$, $\ii(D) = 1$ if and only if $wx-y^2 < 0$, and $\ii(D) = 2$ if and only if $wx-y^2>0$ and $x+w < 0$.

\np\label{Sec7.3} Next we consider the case that $B = A^r$ for $(A,\lambda)$ a principally polarized abelian variety with $\End^0(A) = Z$ a totally imaginary quadratic extension of a totally real field $K = \End^0_\lambda(A)$ with $[K:\QQ]=g$.
We then have
$ \End^0(B) = M_r(Z)$, 
and also, by Theorem \ref{Albert:Classified} Type IV and Corollary \ref{corollary6.2}, that
$$ \End_{\lambda^{\boxtimes r}}^0(B) = \{(\alpha_{ij}) \in M_r(Z) : \alpha_{ij} = \overline{\alpha_{ij}} \},$$
for $\overline{\alpha_{ij}}$ the complex conjugate of $\alpha_{ij}$.
As a consequence it follows that 
$$
\deg(\alpha) = \Nr_{Z/\QQ} (\det(\alpha))
$$
for $\alpha \in \End^0(B)$.

Also, if $\alpha = \Phi_{\lambda^{\boxtimes r}}(D)$ for some $D \in \NS^0(B)$, then 
$$
\frac{(D^g)}{g!} = \mathrm{pNrd}(\alpha) = \Nr_{K/\QQ} (\det(\alpha)).
$$

\np\label{Sec7.4} As in \S \ref{Sec7.2} we can restrict the discussion given in \S \ref{Sec7.3} to the case that $g=1$ and $r=2$.  In particular, $A = E$ an elliptic curve with $\End(E)$ an order in $Z = \QQ(\sqrt{f})$, for $f \in \ZZ_{<0}$; we let $\lambda = 0_E$ the identity of $E$.  We then have
$$\End^0_{\lambda^{\boxtimes 2}}(B) = \left \{ \left( \begin{matrix} x & z+w\sqrt{f} \\
z-w \sqrt{f} & y  \end{matrix} \right)  : x,y,z,w \in \QQ  \right \}.  $$
Also, if 
$$\alpha = \left( \begin{matrix} 
x_{11}+y_{11} \sqrt{f} & x_{12} + y_{12}\sqrt{f} \\
x_{21} + y_{21}\sqrt{f} & x_{22}+y_{22} \sqrt{f}
\end{matrix} \right) \in \End^0(B),$$ then
$\deg (\alpha) = \Nr_{Z/\QQ} (\det (\alpha)),$ while if 
$$\alpha = \left( \begin{matrix} 
x & z+w\sqrt{f} \\
z-w\sqrt{f} & y 
\end{matrix} \right)  \in \End^0_{\lambda^{\boxtimes 2}}(B),
$$ then
\begin{equation}\label{7.8}
p_{D,\lambda}(N) =  N^2 + (x+y)N-z^2+w^2f + xy.
\end{equation}
Considering \eqref{7.8} we deduce that $N = 0$ is a root with multiplicity $2$ if and only if $x=y=w=z=0$.  On the other hand, if $N =0$ is a root of multiplicity $1$, then $\ii(D) = 0$ if and only if $x+y>0$ and $\ii(D) = 1$ if and only if $x+y < 0$.  Similarly, $N =0$ is not a root if and only if  $-z^2+w^2f + xy \not = 0$ and in this case, $\ii(D) = 0$ if and only if $-z^2+w^2f + xy> 0$ and $x+y > 0$, $\ii(D) =1$ if and only if $-z^2+w^2f + xy < 0$, and $\ii(D) =2$ if and only if $-z^2+w^2f + xy > 0$ and $x+ y < 0$.
 
\np\label{} As our final example we suppose that $p := \cchar \kk > 0$ and consider the case that $A = E \times E$ for $E$ a super-singular  elliptic curve.   We then have that $\End(E)$ is an order in the quaternion algebra
$$ \End^0(E) = \mathcal{K} = \QQ \oplus \QQ\mathbf{a} \oplus \QQ \mathbf{b} \oplus \QQ \mathbf{a}\mathbf{b},$$ that the multiplication in $\mathcal{K}$ satisfies the conditions that
$$\mathbf{a}^2,\mathbf{b}^2 \in \QQ, \mathbf{a}^2 < 0, \mathbf{b}^2 < 0, \text{ and } \mathbf{b} \mathbf{a} = - \mathbf{a} \mathbf{b},$$
and also that $\mathcal{K}$ is ramified exactly at $\{p,\infty\}$, see for instance \cite[p.~100-102 and Ex. 3.18]{silverman:2000}.

Additionally, if 
$\alpha = x + y \mathbf{a} + z \mathbf{b} + w \mathbf{a} \mathbf{b} \in \mathcal{K} \text{, with $x,y,z,w \in \QQ$,}$ then
$\alpha' = x - y \mathbf{a} - z \mathbf{b} - w \mathbf{a} \mathbf{b}$,
$ \Trd_{\mathcal{K} / \QQ}(\alpha) = \alpha + \alpha'$,
and 
$\Nrd_{\mathcal{K}/\QQ}(\alpha) = \alpha \alpha' ;$ 
here $'$ denotes the Rosati involution $r_\lambda$ determined by $\lambda = 0_E$ the identity of $E$.
Let
$ \mathbf{a}^2 = a \text{ and } \mathbf{b}^2 = b.$

Now let $R = \End^0(A)=M_2(\mathcal{K})$, and $' : R \rightarrow R$ the Rosati involution $r_{\lambda^{\boxtimes 2}}$.
Keeping with the conventions of \S \ref{Sec2.1}, we have that
$ d^2 = [R:\QQ] = 16, \text{ and }$
$ t=[Z:\QQ]=1$, 
for $Z$ the centre of $R$, and $m = 2$ the Schur index of $\mathcal{K}$.
The field $F=\QQ(\sqrt{b})$ is a splitting field for $R$ and an isomorphism 
\begin{equation}\label{superspecial:isom} 
F \otimes_\QQ R \xrightarrow{\sim} M_4(F)
\end{equation}
is determined by identifying
$$\alpha = \left( 
\begin{matrix} 
x_{11}+y_{11}\mathbf{a} + z_{11}\mathbf{b} + w_{11}\mathbf{a}\mathbf{b} & x_{12}+y_{12}\mathbf{a} + z_{12}\mathbf{b}+w_{12}\mathbf{a}\mathbf{b} \\
x_{21}+y_{21}\mathbf{a} + z_{21} \mathbf{b} + w_{22}\mathbf{a}\mathbf{b} & x_{22}+y_{22}\mathbf{a} + z_{22}\mathbf{b} + w_{22}\mathbf{a} \mathbf{b}
\end{matrix}
\right), $$
with
$$ \alpha = \left( 
\begin{matrix}
x_{11} + z_{11}\sqrt{b} & y_{11}-w_{11}\sqrt{b} & x_{12}+z_{12}\sqrt{b} & y_{12}-w_{12}\sqrt{b} \\
y_{11}a+w_{11}a\sqrt{b} & x_{11}-z_{11}\sqrt{b} & y_{12}a+w_{12}a\sqrt{b} & x_{12}-z_{12}\sqrt{b} \\
x_{21}+z_{21}\sqrt{b} & y_{21}-w_{21}\sqrt{b} & x_{22}+z_{22}\sqrt{b} & y_{22}-w_{21}\sqrt{b} \\
y_{21}a+w_{21} a\sqrt{b} & x_{21}-z_{21}\sqrt{b} & y_{22}a+w_{22}a \sqrt{b} & x_{22}-z_{22}\sqrt{b}
\end{matrix}
\right).$$

By Corollary \ref{corollary6.2}, the elements of $\End^0_{\lambda^{\boxtimes 2}}(A)$ have the form:
\begin{equation}\label{7.9'} \alpha = 
\left( \begin{matrix}
u & x+y\mathbf{a}+z\mathbf{b}+w\mathbf{a}\mathbf{b} \\
x-y\mathbf{a}-z \mathbf{b} - w \mathbf{a} \mathbf{b} & v 
 \end{matrix} \right), 
 \end{equation} for $v,u,w,x,z,y \in \QQ$, and it follows that $\dim_\QQ \End^0_{\lambda^{\boxtimes 2}}(A)=6$.  In addition, using \eqref{superspecial:isom}, we have the identification, for $\alpha$ as in \eqref{7.9'},
$$ \alpha = \\ \left( 
 \begin{matrix}
 u & 0 & x+z\sqrt{b} & y-w\sqrt{b} \\
 0 & u & ya+wa\sqrt{b} & x-z\sqrt{b} \\
 x-z\sqrt{b} & - y + w \sqrt{b} & v & 0 \\
 -ya - w a \sqrt{b} & x+ z \sqrt{b} & 0 & v
 \end{matrix}
 \right), $$
 for $u,v,x,y,z,w \in \QQ$.
 
We also have, for such $\alpha$, that
$$
\Nrd_{R/\QQ}(\alpha) = \det (\alpha) = ({a b^{2} w^{2}-b^{2} z^{2}-a y^{2}-u v+x^{2}})^2
$$
and it follows, using our normalization conventions  \eqref{3.17}, that
\begin{equation}\label{super:singular:euler} \mathrm{pNrd}_\lambda(\alpha) = -w^2ab+z^2b+y^2 a +uv-x^2. \end{equation}
Equation \eqref{super:singular:euler} implies that if $D$ is a divisor on $A$ with 
\begin{equation}\label{7.11}\Phi_{\lambda^{\boxtimes 2}}(D) = \alpha =
\left( \begin{matrix}
u & x+y\mathbf{a}+z\mathbf{b}+w\mathbf{a}\mathbf{b} \\
x-y\mathbf{a}-z \mathbf{b} - w \mathbf{a} \mathbf{b} & v 
 \end{matrix} \right), 
 \end{equation}
 then 
 \begin{equation}\label{7.12} 
 (D^2)/2 = -w^2ab+z^2b+y^2 a +uv-x^2.
 \end{equation}
 In particular, for a divisor $D$ as in \eqref{7.11}, the polynomial $p_{D,\lambda}(N)$ takes the form:
\begin{equation}\label{7.13}
p_{D,\lambda}(N) = N^2 + (u+v)N+uv-x^2-w^2ab+z^2b+y^2a.
\end{equation}
 
Considering \eqref{7.13} we deduce that $N = 0$ is a root with multiplicity $2$ if and only if $u=v=x=y=w=z=0$ and that if $N =0$ is a root of multiplicity $1$, then $\ii(D) = 0$ if and only if $u+v>0$ and $\ii(D) = 1$ if and only if $u+v < 0$.  On the other hand, $N =0$ is a root if and only if  $uv-x^2-w^2ab+z^2b+y^2a \not = 0$ and in this case, $\ii(D) = 0$ if and only if $-w^2ab+z^2b+y^2a+uv-x^2 > 0$ and $u+v > 0$, 
$\ii(D) = 1$ if and only if $-w^2ab+z^2b+y^2a+uv-x^2 < 0$, and 
$\ii(D) = 2$ if and only if $-w^2ab+z^2b+y^2a+uv-x^2 > 0$ and $u+v < 0$.
 
Let also also note that \eqref{7.12} can be used to compute intersection numbers.  Indeed,  suppose that $D_1$ and $D_2$ are divisors on $A$ with 
 $$ \Phi_{\lambda^{\boxtimes 2}} (D_1) = \left(\begin{matrix}
 u_1 & x_1+y_1\mathbf{a}+z_1\mathbf{b} + w_1\mathbf{a}\mathbf{b} \\
 x_1-y_1\mathbf{a}-z_1\mathbf{b}-w_1\mathbf{a} \mathbf{b}  v_1 & v_1
 \end{matrix} \right)$$
 and
 $$ \Phi_{\lambda^{\boxtimes 2}} (D_2) = \left( \begin{matrix}  
 u_2 & x_2+y_2\mathbf{a}+z_2 \mathbf{b} +  w_2 \mathbf{a} \mathbf{b} \\
 x_2-y_2\mathbf{a}-z_2\mathbf{b}-w_2\mathbf{a}\mathbf{b} &  v_2
 \end{matrix}\right).$$
 
If we fix indeterminates $T$ and $S$, we then have
\begin{multline*}
((T D_1+ SD_2)^2)/2 = ( -w_1^2ab+z_1^2b+ay_1^2+u_1v_1-x_1)T^2+ \\ 
(2w_1w_2ab+2z_1z_2b+2y_1y_2a+u_1v_2+u_2v_1-2x_1x_2)ST+ \\
(-w_2^2ab+bz_2^2+ay_2^2+u_2v_2-x_2^2)T^2, 
\end{multline*} 
from which we deduce
$$
D_1.D_2= 2w_1w_2ab+2z_1z_2b+2y_1y_2a+u_1v_2+u_2v_1-2x_1x_2. 
$$

\providecommand{\bysame}{\leavevmode\hbox to3em{\hrulefill}\thinspace}
\providecommand{\MR}{\relax\ifhmode\unskip\space\fi MR }
\providecommand{\MRhref}[2]{%
  \href{http://www.ams.org/mathscinet-getitem?mr=#1}{#2}
}
\providecommand{\href}[2]{#2}


\begin{thebibliography}{10}

\bibitem{Bauer:1998}
T.~Bauer, \emph{On the cone of curves of an abelian variety}, Amer. J. Math.
  \textbf{120} (1998), no.~5, 997--1006.

\bibitem{Chai:Conrad:Oort:2014}
C.-L. Chai, B.~Conrad, and F.~Oort, \emph{Complex multiplication and lifting
  problems}, American Mathematical Society, Providence, RI, 2014.

\bibitem{Conrad:2006}
B.~Conrad, \emph{Chow's {$K/k$}-image and {$K/k$}-trace, and the
  {L}ang-{N}\'eron theorem}, Enseign. Math. (2) \textbf{52} (2006), no.~1-2,
  37--108.

\bibitem{CRI}
C.~W. Curtis and I.~Reiner, \emph{Methods of representation theory. {V}ol.
  {I}}, John Wiley \& Sons, Inc., New York, 1981.

\bibitem{CRII}
\bysame, \emph{Methods of representation theory. {V}ol. {II}}, John Wiley \&
  Sons, Inc., New York, 1987.

\bibitem{Grieve-cup-prod-ab-var}
N.~Grieve, \emph{Index conditions and cup-product maps on abelian varieties},
  Internat. J. Math. \textbf{25} (2014), no.~4, 1450036, 31 pages.

\bibitem{Grieve:theta}
\bysame, \emph{Refinements to {M}umford's theta and adelic theta groups}, Ann.
  Math. Qu\'{e}. \textbf{38} (2014), no.~2, 145--167.

\bibitem{Kani:2016}
E.~Kani, \emph{The moduli spaces of {J}acobians isomorphic to a product of two
  elliptic curves}, Collect. Math. \textbf{67} (2016), no.~1, 21--54.

\bibitem{Katsura:2015}
T.~Katsura, \emph{The {Chern} class map on abelian surfaces}, J. Algebra
  \textbf{425} (2015), 85--106.

\bibitem{Mar}
D.~Marcus, \emph{Number fields}, Springer-Verlag, 1977.

\bibitem{Milne:1986}
J.~S. Milne, \emph{Abelian varieties}, Arithmetic geometry ({S}torrs, {C}onn.,
  1984), Springer, New York, 1986, pp.~103--150.

\bibitem{MV}
B.~Moonen and G.~van~der Geer, \emph{Abelian varieties}, Electronic pre-print,
  2012.

\bibitem{Mum:v1}
D.~Mumford, \emph{Abelian varieties}, Published for the Tata Institute of
  Fundamental Research, Bombay; Oxford University Press, London, 1970.

\bibitem{Mum:Quad:Eqns}
\bysame, \emph{Varieties defined by quadratic equations}, Questions on
  {A}lgebraic {V}arieties ({C}.{I}.{M}.{E}., {III} {C}iclo, {V}arenna, 1969),
  Edizioni Cremonese, Rome, 1970, pp.~29--100.

\bibitem{Oort:1988}
F.~Oort, \emph{Endomorphism algebras of abelian varieties}, Algebraic geometry
  and commutative algebra, {V}ol.\ {II}, Kinokuniya, Tokyo, 1988, pp.~469--502.

\bibitem{Oort:VanDerPut:1988}
F.~Oort and M.~van~der Put, \emph{A construction of an abelian variety with a
  given endomorphism algebra}, Compositio Math. \textbf{67} (1988), no.~1,
  103--120.

\bibitem{Reiner:2003}
I.~Reiner, \emph{Maximal orders}, The Clarendon Press, Oxford University Press,
  Oxford, 2003.

\bibitem{Shimura}
G.~Shimura, \emph{On analytic families of polarized abelian varieties and
  automorphic functions}, Ann. of Math. (2) \textbf{78} (1963), 149--192.

\bibitem{Shin}
S.~W. Shin, \emph{Abelian varieties and {W}eil representations}, Algebra Number
  Theory \textbf{6} (2012), no.~8, 1719--1772.

\bibitem{silverman:2000}
J.~H. Silverman, \emph{The arithmetic of elliptic curves}, Springer, Dordrecht,
  2009.

\end{thebibliography}
\end{document}